\documentclass[french,12pt]{amsart}

\usepackage{graphicx}
\usepackage[francais]{babel}
\usepackage[latin1]{inputenc}
\usepackage[T1]{fontenc}
\usepackage{amsmath}
\usepackage{hyperref} 

\theoremstyle{plain}
\newtheorem{thm}{Théorème}
\newtheorem{lem}[thm]{Lemme}
\newtheorem{cor}[thm]{Corollaire}
\newtheorem{prop}[thm]{Proposition}
\newtheorem{crit}[thm]{Critère}

\theoremstyle{definition}

\newtheorem{ex}[thm]{Exemple}
\newtheorem{conv}[thm]{Convention}
\newtheorem{exo}[thm]{Exercice}
\newtheorem{rem}[thm]{Remarque}
\newtheorem{defi}[thm]{Définition}

\newtheorem{nota}[thm]{Notation}

\newcommand{\bgt}{\begin{thm}}
\newcommand{\et}{\end{thm}}
\newcommand{\bgp}{\begin{prop}}
\newcommand{\ep}{\end{prop}}
\newcommand{\bgd}{\begin{defi}}
\newcommand{\ed}{\end{defi}}
\newcommand{\bgl}{\begin{lem}}
\newcommand{\el}{\end{lem}}
\newcommand{\bgr}{\begin{rem}}
\newcommand{\er}{\end{rem}}
\newcommand{\bge}{\begin{ex}}
\newcommand{\ee}{\end{ex}}
\newcommand{\bgex}{\begin{exo}}
\newcommand{\eex}{\end{exo}}
\newcommand{\bgi}{\begin{itemize}}
\newcommand{\ei}{\end{itemize}}
\newcommand{\bgpe}{\begin{proof}}
\newcommand{\epe}{\end{proof}}
\newcommand{\bgar}{\begin{array}}
\newcommand{\ear}{\end{array}}
\newcommand{\bgn}{\begin{nota}}
\newcommand{\en}{\end{nota}}
\newcommand{\bgc}{\begin{cor}}
\newcommand{\ec}{\end{cor}}
\newcommand{\bgen}{\begin{enumerate}}
\newcommand{\een}{\end{enumerate}}

\newcommand{\mn}{{\mathbb N}}

\newcommand{\mz}{{\mathbb Z}}
\newcommand{\mfsep}{{F_{s\acute{e}p}}}

\DeclareMathOperator{\CH}{CH} 

\DeclareMathOperator{\SB}{SB}

\DeclareMathOperator{\Mot}{Mot}

\DeclareMathOperator{\mult}{mult}

\newcommand{\SO}{\mathrm{SO}}

\DeclareMathOperator{\Sl}{SL}

\DeclareMathOperator{\Spec}{Spec}

\title{Équivalence motivique des groupes algébriques semisimples}

\author{Charles De Clercq}

\date{\today}

\begin{document}

\parindent=0pt

\begin{abstract}
Deux groupes semisimples sont dits motiviquement équivalents si les motifs des variétés de drapeaux généralisées associées sont isomorphes modulo tout nombre premier $p$. L'objet de cette note est de construire les invariants combinatoires qui caractérisent l'équivalence motivique et sont les analogues motiviques des indices de Tits apparaissant dans la classification des groupes algébriques semisimples. L'expression de ces invariants -les $p$-indices de Tits supérieurs- en fonction des indices classiques associés aux structures naturelles sous-jacentes aux groupes semisimples permet de produire des critères algébriques d'équivalence motivique, généralisant le critère de Vishik d'équivalence motivique des quadriques. Elle permet en outre de clarifier le lien qu'entretiennent les motifs et la géométrie birationnelle des variétés de drapeaux.
\end{abstract}

\bigskip
\maketitle


\section{Introduction}

La classification des groupes algébriques semisimples a été obtenue dans les années soixantes par les efforts conjugués de Borel, Tits, Satake, Serre, Springer et bien d'autres, dans la lignée de la classification de Chevalley sur un corps algébriquement clos. Comme pressenti par Siegel et Weil pour les groupes classiques, ces groupes sont intimement liés aux structures algébriques qui vivent sur le corps de base (formes quadratiques, algèbres simples centrales à involution, algèbres de Jordan, de Cayley...). Un des principaux outils de cette classification est la notion d'indice de Tits d'un groupe semisimple, qui consiste en son diagramme de Dynkin muni d'une action du groupe de Galois absolu, et coloré d'une manière spécifique. La détermination de l'ensemble des indices de Tits possibles, obtenue dans \cite{tits1}, peut être vue comme un prolongement des théorèmes de Witt de la théorie algébrique des formes quadratiques.

L'objectif de cet article est de fournir une classification des groupes semisimples en fonction d'invariants de nature motivique. Comme il est d'usage dans la théorie classique, on ne souhaite imposer aucune restriction sur le corps de définition des groupes considérés -notamment sur sa caractéristique- le cadre choisi est ainsi celui des motifs de Grothendieck construits à partir des groupes de Chow. L'invariant naturel que l'on associe à un groupe semisimple est alors l'ensemble des motifs des variétés de drapeaux généralisées qui lui sont associées.

Rappelons qu'un groupe semisimple $G$ sur un corps $F$ étant fixé, une $G$-variété $X$ est une \emph{$G$-variété de drapeaux généralisée} s'il existe un corps $E$ contenant $F$ sur lequel elle devient isomorphe au quotient de $G_E$ par un sous-groupe parabolique. L'étude des motifs de variétés de drapeaux généralisées, initiée par les travaux de Rost et Vishik sur le motif des quadriques dans les années 90 \cite{rost1,rost2,vishmot} est un domaine florissant qui a permis la résolution de nombreuses conjectures classiques et difficiles (citons par exemple les conjectures de Hoffmann, de Kaplansky, ou encore de Milnor).

De la même manière que pour la classification de Tits, il s'agit de construire des invariants combinatoires qui caractérisent les invariants motiviques des groupes algébriques semisimples. Petrov, Semenov et Zainoulline introduisent dans \cite{jinv} le $J$-invariant, un invariant numérique généralisant le $J$-invariant des formes quadratiques de Vishik. Le $J$-invariant d'un groupe semisimple \emph{intérieur} $G$ décrit le motif de la variété des sous-groupes de Borel de $G$ et possède de nombreuses applications (voir \cite{shells} ou \cite{trial}). Il est aussi un des ingrédients essentiels dans la classification des variétés de drapeaux généralisées \emph{génériquement déployées}, obtenue dans \cite{gensplit1}, \cite{gensplit2}.

Un conséquence essentielle de la construction de ces invariants combinatoires est qu'elle permet de relier l'étude des motifs aux invariants classiques associés aux groupes semisimples. Le célèbre critère de Vishik \cite[Theorem 4.18]{vish} (voir \cite{kahn2,kahn} pour un panorama général de la théorie de Vishik) stipule que les motifs de deux quadriques projectives de même dimension sont isomorphes à coefficients entiers si et seulement si les indices de Witt des formes quadratiques associées coincident sur toutes les extensions de leur corps de définition. Initialement démontré en caractéristique $0$ puis étendu par Karpenko en caractéristique positive \cite{motequivkarp}, ce résultat montre que le motif d'une quadrique est essentiellement déterminé par l'isotropie de la forme quadratique associée sur les extensions du corps de base. Nous montrons par la suite comment la classification que nous obtenons permet non seulement d'obtenir une nouvelle preuve du critère de Vishik, mais aussi de le généraliser à n'importe quel groupe semisimple $p$-intérieur -sans donc exclure les groupes \emph{extérieurs}-. L'existence de ces critères algébriques d'équivalence motivique permet en particulier d'éclaircir le lien qu'entretiennent les motifs et la géométrie birationnelle des variétés de drapeaux généralisées.

Si $G$ est un groupe semisimple et $\Lambda$ est un anneau commutatif fixé, on associe en section VI à tout sous-ensemble $\Theta$ des sommets de son diagramme de Dynkin un motif. Ce motif est apppelé le \emph{motif standard} de $G$ de type $\Theta$ (à coefficients dans $\Lambda$). Lorsque $G$ est de type intérieur, le motif standard de type $\Theta$ de $G$ est simplement le motif à coefficients dans $\Lambda$ d'une $G$-variété de drapeaux généralisée de type $\Theta$, au sens de Borel-Tits \cite{boreltits}. 

\begin{defi}
Soient $G$ et $G'$ deux groupes semisimples sur un corps $F$, tous deux formes intérieures d'un même groupe quasi-déployé. On dit que $G$ et $G'$ sont motiviquement équivalents à coefficients dans un anneau $\Lambda$ si pour tout sous ensemble $\Theta$ de leur diagramme de Dynkin, les motifs standards de type $\Theta$ à coefficients dans $\Lambda$ de $G$ et $G'$ sont isomorphes. Si $p$ est un nombre premier et $\Lambda$ est un corps à $p$ éléments, on dira que $G$ et $G'$ sont motiviquement équivalents modulo $p$.
\end{defi}

On dit simplement que deux groupes semisimples sont \emph{motiviquement équivalents} s'ils sont motiviquement équivalents modulo tout nombre premier $p$. Comme les résultats de cette note le montrent -le corollaire \ref{inner} , pour ne citer que lui- l'équivalence motivique ainsi définie se révèle beaucoup plus riche et féconde que l'équivalence motivique à coefficients entiers. Historiquement introduite comme une contrainte nécessaire à l'existence des décompositions motiviques, l'étude des motifs à coefficients finis montre ici toute son importance. En effet, les motifs à coefficients entiers se révèlent -dès lors qu'on s'écarte des groupes othogonaux- insuffisants pour fournir des invariants adaptés à l'étude de la géométrie birationnelle des variétés de drapeaux. Nous verrons lors des sections IX et X notamment les conséquences de cette classification lors de l'étude de l'équivalence motivique des groupes spéciaux linéaires.\\

Un des outils essentiels de notre classification est le théorème \ref{classgeoalg} qui relie l'équivalence motivique et l'équivalence modulo $p$ des variétés projectives, notion introduite en section II. La preuve de ce résultat repose sur deux notions fondamentales. D'une part elle tire parti des décompositions motiviques de type Levi des variétés de drapeaux généralisées isotropes, dues à Chernousov, Gille et Merkurjev \cite{chermergil} (voir aussi \cite{brosnan}). L'autre ingrédient essentiel utilisé est la notion d'extension pondérée du corps de base, introduite en section V. La section VI est dédiée à la construction des $p$-indices de Tits supérieurs, objets combinatoires qui correspondent aux versions locales des indices de Tits classiques. L'objet de la section IX est de prouver comment ces données combinatoires caractérisent les classes d'équivalence motivique des groupes semisimples.

La détermination obtenue dans \cite{decskip} de toutes valeurs possibles des $p$-indices de Tits ainsi que leur connexions avec les invariants classiques des groupes semisimples permettent de généraliser le critère de Vishik aux groupes semisimples et relie les motifs et la géométrie birationnelle des variétés de drapeaux généralisées. Ces critères de nature algébrique mettent en jeu les invariants classiques des structures associées aux groupes semisimples (formes quadratiques pour les groupes orthogonaux, algèbres centrales simples pour les groupes spéciaux linéaires...) et n'ont pas d'analogue si l'on considère l'équivalence motivique à coefficients entiers. En somme, l'invariant motivique adapté à la géométrie birationnelle des variétés de drapeaux généralisées n'est pas le motif à coefficients entiers, mais plutôt la collection de tous les motifs à coefficients dans $\mathbb{F}_p$, pour tout nombre premier $p$.

\section{Dominance, équivalence modulo $p$ des variétés projectives}

Dans toute la suite on fixe un corps de base $F$, et $p$ désignera toujours un nombre premier. Pour tout nombre premier $p$ on introduit une relation de préordre sur la classe des $F$-variétés propres : la relation de \emph{dominance modulo $p$}. Suivant \cite{EKM,fulton}, on note par $\CH_k(X)$ le groupe de Chow des cycles algébriques de dimension $k$ sur une $F$-variété $X$ modulo équivalence rationelle.

Soit $\alpha:X\leadsto Y$ une correspondance de Chow de degré $0$ entre $F$-variétés, $Y$ étant supposée propre -c'est à dire un élément du groupe de Chow $\CH_{\dim(X)}(X\times Y)$-. Si $\pi$ est la projection naturelle de $X\times Y$ sur le premier facteur, la \emph{multiplicité} de $\alpha$ est l'unique entier $\mult(\alpha)$ satisfaisant la relation $\pi_{\ast}(\alpha)=\mult(\alpha)\cdot [X]$. 

\begin{defi}
Soient $X$ et $Y$ deux $F$-variétés propres. On dit que \emph{$X$ domine $Y$} s'il existe une correspondance de Chow $X\leadsto Y$ de degré $0$ et de multiplicité $1$. S'il existe une correspondance $\alpha:X\leadsto Y$ de degré $0$ et de multiplicité première à $p$, on dit que \emph{$X$ domine $Y$ modulo $p$}.
\end{defi}

Un calcul montre que la multiplicité de la composée  $\alpha\circ\beta$ de deux correspondances (au sens de \cite[§2]{manin}) est le produit des multiplicités respectives de $\alpha$ et $\beta$. La dominance définie précédemment est donc une relation de préordre sur la classe des $F$-variétés propres, et la relation d'équivalence qui lui est associée sera simplement appelée relation d'\emph{équivalence}. Ces considérations valent bien entendu aussi pour les relations de dominance modulo $p$, et on parle alors d'\emph{équivalence modulo $p$} pour les relations d'équivalences associées.\\

Les relations d'équivalence et d'équivalence modulo $p$ sont étroitement liées à la géométrie birationnelle des variétés. En effet la relation d'équivalence coincide avec la relation d'équivalence birationnelle stable pour les variétés de Severi-Brauer, tout comme pour nombre d'autres variétés de drapeaux générlisées. Notons que la situation est encore plus ténue entre deux quadriques projectives, qui sont stablement birationellement équivalentes si et seulement si elles sont équivalentes modulo $2$, en vertu du théorème de Springer \cite{springer}.

\section{Classifications de Tits et Borel-Tits}

Le lecteur non familier avec la classification des groupes semisimples et la classification de Borel-Tits des variétés de drapeaux généralisées est fortement invité à consulter \cite{borel}, \cite{boreltits}, \cite{KMRT},  \cite{index} \cite{index2}, \cite{tits1}, \cite{weil} pour plus de détails.

Comme d'accoutumée, on considère dans un premier temps un groupe algébrique semisimple $G$ sur un corps $F$ supposé séparablement clos. Un tore maximal $T\subset G$ fixé donne lieu à un système de racines $\Phi(G,T)$ au sens de \cite{bourba}, et donc à un diagramme de Dynkin, une fois un système de racines simples choisi. Les sous-tores maximaux de $G$ étant tous conjugués, ce diagramme de Dynkin ne dépend pas du choix de $T$, et il ne dépend pas non plus du choix du système de racines simples. Le diagramme de Dynkin d'un groupe semisimple $G$, que l'on note $\Delta(G)$, caractérise complètement sa classe d'isogénie stricte. Le nombre de sommets de $\Delta(G)$ est le \emph{rang} du groupe $G$.

Considérons désormais le cas où $F$ est quelconque, et notons en $\mfsep$ une clôture séparable. En vertu de \cite[Exposé XIV, Théorème 1.1]{groth} un groupe semisimple $G$ sur $F$ possède toujours un tore maximal, qui n'est cependant cette fois pas nécessairement déployé. On associe à $G$ le système de racine $\Phi(G_{\mfsep},T_{\mfsep})$, ainsi que le diagramme de Dynkin $\Delta(G_{\mfsep})$, que l'on note à nouveau $\Delta(G)$. Le groupe de Galois absolu $Gal(\mfsep/F)$ agit naturellement sur le système de racines $\Phi(G_{\mfsep},T_{\mfsep})$.\\

Borel et Tits \cite{boreltits} introduisent une seconde action du groupe de Galois absolu de $F$ sur le diagramme de Dynkin de $G$. Le groupe de Weil agissant simplement transitivement sur les systèmes de racines simples de $\Phi(G_{\mfsep},T_{\mfsep})$, il existe pour tout tel système $\Pi$ un élément $w_{\sigma}$ qui le stabilise. L'action associée sur le diagramme de Dynkin est appelée la $\ast$-action.

Les classes d'isomorphisme des $G_{\mfsep}$-variétés de drapeaux sont en bijection avec les avec les sous-ensembles (des sommets) du diagramme de Dynkin de $G_{\mfsep}$, par \cite{boreltits}. Si $\Theta$ est le sous-ensemble associé à une $G_{\mfsep}$-variétés de drapeaux $X$, on dit que $X$ est de \emph{type $\Theta$}. Par descente \cite{borelspringer}, les classes d'isomorphismes de $G$-variétés de drapeaux généralisées sont alors en correspondance bijective avec les sous-ensembles de $\Delta(G)$ invariants pour la $\ast$-action. Les constructions et définitions ultérieures étant toutes visiblement invariantes par isomorphisme, nous noterons dans la suite $X_{\Theta,G}$ toute $G$-variété de drapeaux généralisée de type $\Theta$, identifiant ces variétés et leur classe d'isomorphisme.

\begin{conv}
On fixe par la suite pour convention que la variété des sous-groupes de Borel de $G$ est de type $\Delta(G)$. Une $\Sl_n$-variété (resp. une $\SO_n$-variété) de drapeau généralisée de type $\Theta=\{k\}$ correspond alors à une Grassmannienne de $k$-plans dans un espace de dimension $n$ (resp. une variété de $k$-plans isotropes). Dans ce contexte, la variété de Severi-Brauer d'une algèbre simple centrale $A$ est alors de type $\{1\}$ pour le groupe spécial linéaire de $A$ (on renvoit à la section X pour plus de détails).
\end{conv}

Le \emph{type} d'un groupe algébrique semisimple $G$ est la donnée de son diagramme de Dynkin, muni de la $\ast$-action du groupe de Galois absolu sur l'ensemble de ses sommets. Deux groupes semisimples sont de même type si et seulement s'ils sont tous deux formes intérieures d'un même groupe quasi déployé.

\begin{defi}
Soient $G$ et $G'$ deux groupes semisimples sur $F$, tous deux formes intérieures d'un même groupe quasi-déployé. On dit que $G$ et $G'$ sont \emph{équivalents modulo $p$} si pour tout sous-ensemble $\ast$-invariant $\Theta$ de leur diagramme de Dynkin, leurs variétés de drapeaux généralisées respectives sont équivalentes modulo $p$.
\end{defi}

Notons que les classes équivalence modulo $p$ des groupes semisimples sont bien connues lorsque $F$ est séparablement clos. En effet sous cette hypothèse un groupe semisimple sur $F$ est \emph{déployé} (c'est à dire qu'il possède un tore maximal déployé). Les variétés de drapeaux généralisées sont alors rationnelles, et deux groupes semisimples de même type sont équivalents modulo tout nombre premier $p$.

\section{Indices de Tits, sous groupes de Levi}

Les indice de Tits sont les données combinatoires qui déterminent l'isotropie des groupes semisimples. L'indice de Tits d'un groupe semisimple $G$ est la donnée du diagramme de Dynkin de $G$ muni de la $\ast$-action, ainsi que d'un sous-ensemble $\Delta_0(G)$ des sommets de ce diagramme.

La coloration des sommets du diagramme de Dynkin de $G$ est déterminée comme suit. On fixe un tore maximal $T$ de $G$, ainsi qu'un tore \emph{déployé} maximal $S$ contenu dans $T$. L'ensemble des racines simples de $\Phi(G_{\mfsep},T_{\mfsep})$ qui s'annulent sur $S$ correspond à un sous-ensemble des sommets du diagramme de Dynkin de $G$ que l'on note $\delta_0$. L'ensemble $\delta_0$ est celui des \emph{orbites distinguées} de $\Delta(G)$ et correspond au complémentaire de $\Delta_0$ (voir \cite{tits1}). On dit que le groupe $G$ est \emph{quasi-déployé} lorsque $\delta_0(G)=\Delta(G)$, tandis dans le cas extrême inverse -c'est à dire lorsque $\delta_0(G)$ est vide- $G$ est dit \emph{anisotrope}. La détermination de l'ensemble des valeurs possibles des indices de Tits possibles est présentée dans \cite{tits1}.

L'indice de Tits d'un groupe semisimple est d'autant plus important qu'il détermine l'isotropie (c'est à dire l'existence de points rationnels) des variétés de drapeaux généralisées. On a en effet le résultat \cite{borelspringer} de type Witt suivant : une $G$-variété de drapeaux généralisée possède un point rationnel si et seulement si le sous-ensemble qui lui est associé est contenu dans $\delta_0(G)$.

Nous aurons enfin besoin de la construction suivante, qui généralise celle de \emph{noyau anisotrope} de $G$, et que nous utiliserons pour ramener la preuve du théorème \ref{classgeoalg} au cas anisotrope. Soit $G$ une groupe semisimple possédant un sous-groupe parabolique $P$ de type $\Theta$ défini sur $F$. Suivant \cite{borel} (voir aussi \cite{boreltits},  \cite{kerreh}, \cite{tits1}) on note $G_{\Theta}$ la partie semisimple d'un sous-groupe de Levi de $P$. Le type du groupe $G_{\Theta}$ est obtenu à partir du type de $G$ en lui ôtant les sommets qui appartiennent à $\Theta$, ainsi que toutes les arrêtes du diagramme de Dynkin de $G$ qui atteignent ces sommets.

\begin{prop}\label{prop5}
Soient $G$ et $G'$ deux groupes semisimples sur $F$, tous deux formes intérieures d'un même groupe quasi-déployé. Les groupes $G$ et $G'$ sont équivalents modulo $p$ si et seulement si pour tout sous-ensemble $\ast$-invariant $\Theta$ de leur diagramme de Dynkin contenu dans $\delta_0$, les groupes $G_{\Theta}$ et $G'_{\Theta}$ sont équivalents modulo $p$.
\end{prop}

\begin{proof}
On s'attarde sur la condition nécessaire, la condition suffisante étant claire. On fixe un sous-ensemble $\ast$-invariant $\Theta$ ainsi qu'un autre sous ensemble $\ast$-invariant $\tilde{\Theta}$ de $\Delta\setminus \Theta$.\\
Par hypothèse, les variétés de drapeaux généralisées $X_{\tilde{\Theta},G}$ et $X_{\tilde{\Theta},G'}$ sont équivalentes modulo $p$. En outre par \cite[Theorem 3.18]{kerreh}, le corps des fonctions de la variété $X_{\tilde{\Theta},G'}$ est une extension transcendante pure de celui de $X_{\tilde{\Theta},G_{\Theta}'}$. L'invariance par homotopie des groupes de Chow \cite[Theorem 57.13]{EKM} assure donc de l'existence d'un zéro cycle de degré premier à $p$ sur la variété $X_{\tilde{\Theta},G}$ après extension au corps des fonctions de $X_{\tilde{\Theta},G_{\Theta}'}$. On voit donc que la variété $X_{\tilde{\Theta},G_{\Theta}'}$ domine $X_{\tilde{\Theta},G}$ modulo $p$, par \cite[Lemma 75.1]{EKM}.\\
Ainsi, $X_{\tilde{\Theta},G}$ possède un point fermé sur le corps des fonctions de $X_{\tilde{\Theta},G_{\Theta}'}$ dont le corps résiduel, qu'on note $E$, est une extension finie de degré premier à $p$. Le corps des fonctions de la $G_E$-variété de drapeaux généralisée $X_{\tilde{\Theta},G_E}$ est donc une extension transcendante pure de $E$. Invoquant à nouveau l'invariance par homotopie des groupes de Chow, $X_{\tilde{\Theta},G_{\Theta}}$ a un zéro-cycle de degré premier à $p$ sur $E$, puisqu'elle en possède un sur le corps des fonctions de $X_{\tilde{\Theta},G_E}$. On en déduit que $X_{\tilde{\Theta},G_{\Theta}'}$ domine $X_{\tilde{\Theta},G_{\Theta}}$ modulo $p$ et le même raisonnement interchangeant leur rôles induit que ces variétés sont équivalentes modulo $p$.
\end{proof}

\section{Extensions pondérées}

Deux variétés équivalentes modulo un nombre premier $p$ le demeurent sur n'importe quelle extension du corps de base. On s'intéresse dans cette partie via la notion d'extension pondérée au problème inverse, c'est à dire aux extensions vérifiant des propriétés de descente pour l'équivalence modulo $p$.

Un groupe semisimple $G$ est dit \emph{de type intérieur} s'il est forme intérieure d'un groupe déployé. Il existe une extension finie, minimale et Galoisienne -définie de manière unique à isomorphisme près- sur laquelle $G$ devient de type intérieur. Si le degré de cette extension est une puissance d'un nombre premier $p$, on dit alors que $G$ est \emph{$p$-intérieur} (notons par exemple qu'un groupe intérieur est alors $p$-intérieur pour tout nombre premier $p$).

Dans toute la suite de cette section, on fixe un groupe semisimple $p$-intérieur $G$ et on note $F_{int}/F$ une extension minimale sur laquelle $G$ est de type intérieur. On dit qu'une $F$-variété $X$ est une \emph{quasi-$G$-variété de drapeaux généralisée} s'il existe une tour d'extensions $F_{int}/L/F$ et une $G_L$-variété de drapeaux généralisée $Y$ dont $X$ est la corestriction au corps de base.

\begin{defi}
Soit $X$ une $G$-variété de drapeaux généralisée. Une extension $E/F$ est \emph{pondérée pour $X$ modulo $p$} si pour toutes quasi-$G$-variétés de drapeaux $Y$ et $Y'$ dominant $X$ modulo $p$, $Y$ et $Y'$ sont équivalentes modulo $p$ si et seulement si $Y_E$ et $Y'_E$ le sont.
\end{defi}

On dit simplement qu'une extension $E/F$ est \emph{pondérée} pour $X$ si elle l'est modulo tout nombre premier. Bien que moins élémentaire, le contexte dans lequel la notion d'extension pondérée est la plus limpide est celui des \emph{motifs supérieurs} : les extensions pondérées pour $X$ sont les extensions qui respectent les motifs indécomposables apparaissant dans la décomposition du motif de $X$ (on renvoit à la section suivante pour plus de détails). La classe des extensions pondérées pour une variété $X$ modulo $p$ est visiblement stable par sous-extension et par extensions successives. Une extension unirationnelle est pondérée pour toute variété $X$, en vertu de l'invariance par homotopie des groupes de Chow. Le résultat suivant est essentiellement contenu dans \cite{kar1}.

\begin{prop}\label{pond}
Le corps des fonctions d'une $G$-variété de drapeaux généralisée $X$ est une extension pondérée pour $X$.
\end{prop}

\begin{proof}
Fixons un nombre premier $l$ et considérons deux quasi-$G$-variétés de drapeaux généralisées $Y$ et $Y'$ qui dominent $X$ modulo $l$. La variété $X$ possède sur le corps des fonctions de $Y$ un point fermé dont le corps résiduel $E$ est une extension de degré premier à $l$. Le corps des fonctions de la $G_E$-variété de drapeaux généralisée $X_E$ est donc une extension transcendante pure de $E$.

Supposons de plus que $Y$ et $Y'$ soient équivalentes modulo $l$ sur le corps des fonctions de $X$. Par hypothèse, la variété $Y'$ possède un zéro-cycle de degré premier à $l$ sur l'extension composée des corps de fonctions de $X$ et $Y$, donc a fortiori sur le corps des fonctions de $X_E$. L'invariance par homotopie des groupes de Chow implique donc que $Y'$ possède un zéro-cycle de degré premier à $l$ sur $E$, c'est à dire que $Y$ domine $Y'$ modulo $l$. Le même raisonnement les rôles de $Y$ et $Y'$ étant intervertis permet de conclure.
\end{proof}

\section{Motifs, décompositions}

On s'appuie principalement sur \cite{EKM,manin} pour ces quelques rappels sur la catégorie $\Mot(F,\Lambda)$ des motifs de Grothendieck construits à l'aide des groupes de Chow à coefficients dans un anneau commutatif $\Lambda$.

Dans un premier temps on considère la catégorie des correspondances, qui est obtenue à partir de celle des $F$-variétés projectives lisses en remplaçant les morphismes de variétés par les correspondances graduées à coefficients dans $\Lambda$, munies de leur composition habituelle. La catégorie ainsi obtenue est préadditive, et la catégorie des motifs $\Mot(F,\Lambda)$ est obtenue à partir de celle-ci par complétion additive, puis Karoubienne.

La catégorie $\Mot(F,\Lambda)$ est tensorielle, et ses objets sont donc des sommes directes finies de triplets $(X,\pi)[i]$, où $X$ est une $F$-variété projective lisse irréductible, $\pi$ est un projecteur des endomorphismes de $X$, c'est à dire une correspondance de degré $0$ de $X$ vers lui-même idempotente, et $i$ est un entier. Le motif $M(X)$ d'une variété $X$ est par définition le triplet $(X,\Gamma_{id_X})[0]$, où $\Gamma_{id_X}$ est la classe du graphe de l'identité de $X$. Pour alléger les notations on omet très régulièrement le graphe de l'identité en mentionnant le motif de $X$ par la suite.

Un exemple fondamental est celui des \emph{motifs de Tate}, qui sont les motifs décalés du point. Comme d'usage, un entier $i$ étant fixé, le motif de Tate $\Spec(F)[i]$ sera noté $\Lambda[i]$.

\begin{defi}
Soit $G$ un groupe algébrique semisimple, $\Theta$ un sous ensemble des sommets de son diagramme de Dynkin et $F_{\Theta}/F$ une extension minimale sur laquelle $\Theta$ est $\ast$-invariant. Le \emph{motif standard de $G$ de type $\Theta$ à coefficients dans $\Lambda$}, noté $M_{\Theta,G}$, est le motif à coefficients dans $\Lambda$ de la corestriction de la variété $X_{\Theta,G_{F_{\Theta}}}$ au corps de base.
\end{defi}

Si $\Theta$ est invariant sous l'action du groupe de Galois -c'est le cas par exemple lorsque $G$ est de type intérieur- le motif standard de type $\Theta$ de $G$ n'est autre que le motif d'une $G$-variété de drapeaux généralisée de ce type. Dans la suite on omet fréquemment de préciser l'anneau des coefficients des motifs standard que l'on considère, dans un soucis de légèreté. Ce choix est primordial, ne serait-ce que pour utiliser la propriété de Krull-Schmidt, c'est à dire l'existence et l'unicité des décompositions motiviques. En effet d'après Chernousov et Merkurjev \cite[Example 32]{chermer} (voir aussi le \cite[Corollary 2.7]{chowmottwistflag}) cette propriété n'est par exemple pas vérifiée à coefficients entiers. Ces décompositions motiviques sont fondamentales pour l'étude de l'équivalence motivique et requièrent certaines hypothèses -notamment de finitude sur l'anneau des coefficients- comme le montrent \cite{chermer,upper}. 

Notons que si l'on néglige la torsion des groupes de Chow, par exemple en considérant les motifs à coefficients rationnels, tous les motifs standards sont \emph{déployés}, c'est à dire somme directes de motifs de Tate. Deux groupes semisimples de même type sont ainsi toujours motiviquement équivalents à coefficients rationnels, c'est une conséquence des décompositions motiviques obtenues par Köck \cite{kock}. 

Fixons désormais un nombre premier $p$ et notons $\Mot_G(F,\mathbb{F}_p)$ la sous-catégorie épaisse de $\Mot(F,\mathbb{F}_p)$ engendrée par les facteurs directs des motifs de quasi-$G$-variétés de drapeaux généralisées. En vertu de \cite[Corollary 35]{chermer}, \cite[Corollary 2.6]{upper} la catégorie $\Mot(F,\mathbb{F}_p)$ est de Krull-Schmidt.

Notons que par \cite{decvarprojcoeff,vishyag}, tous les résultats obtenus dans cet article à coefficients dans $\mathbb{F}_p$ s'étendent à tout anneau fini, local et dont le corps résiduel est de caractéristique $p$. Il en va en fait de même pour tout corps de caractéristique $p$, sous réserve de Krull-Schmidt.

\section{Chirurgie des motifs de variétés de drapeaux}

On fixe désormais un nombre premier $p$, et tous les groupes semisimples considérés par la suite seront suppposés $p$-intérieurs. Comme vu précédemment, le motif à coefficients dans un corps à $p$ éléments d'une variété de drapeaux généralisée admet une décomposition essentiellement unique. La théorie des motifs supérieur a pour objectif de décrire et classifier les motifs indécomposables qui apparaissent dans ces décompositions.\\
Si $X$ est une quasi-$G$-variété de drapeaux généralisée, il existe à isomorphisme près un unique facteur indécomposable du motif de $X$ à coefficient dans $\mathbb{F}_p$ qui contient le motif de Tate $\mathbb{F}_p[0]$ après extension des scalaires. Ce motif est noté $U_X$ et appelé le \emph{motif supérieur} de $X$.

Si $\Theta$ est un sous-ensemble des sommets du diagramme de Dynkin d'un groupe semisimple $G$, le \emph{motif supérieur standard de $G$ de type $\Theta$}, noté $U_{\Theta,G}$, est le motif supérieur du motif standard $M_{\Theta,G}$. Le théorème de structure de Karpenko \cite[Theorem 1.1]{outer} stipule que l'ensemble $U^p_G$ des motifs supérieurs standards décalés de $G$ à coefficients dans $\mathbb{F}_p$ décrit tous les motifs indécomposables de $\Mot_G(F;\mathbb{F}_p)$.

\begin{defi}
Soit $p$ un nombre premier, $G$ un groupe semisimple $p$-intérieur et $M$ un objet de $\Mot_G(F;\mathbb{F}_p)$. L'\emph{application caractéristique} $\chi_{M}:U^p_G\times \mz\rightarrow \mn$ de $M$ est l'application qui associe à tout couple $(U_{\Theta,G},i)$ le nombre de facteurs directs de $M$ isomorphes à $U_{\Theta,G}[i]$.
\end{defi}

Tout motif $M$ de $\Mot_G(F,\mathbb{F}_p)$ est visiblement déterminé de manière unique à isomorphisme près par son application caractéristique. Pour tout entier $i$, la \emph{tranche supérieure à $i$} de $M$, notée $M^{\geq i}$, est le motif dont l'application caractéristique prend les mêmes valeurs que $\chi_M$ pour tout couple $(U_{\Theta,G},j)$ lorsque $i$ est inférieur ou égal à $j$, et nulle partout ailleurs. On définit de manière similaire la \emph{tranche inférieure à $i$} de $M$, ainsi que les tranches supérieures et inférieures strictes à $i$ de $M$ en remplacant l'inégalité $i\leq j$ par $j\leq i$, $i>j$ et $i<j$, respectivement. Ces motifs sont sans surprise notés respectivement $M^{\leq i}$, $M^{>i}$ et $M^{< i}$.\\

Nous l'avons vu précédemment, la notion d'extension pondérée fournit une classe d'extensions de corps qui respecte les classes d'équivalence des quasi-$G$-variétés de drapeaux. Ces classes d'équivalences sont en correspondance bijective avec les classes d'isomorphismes de motifs supérieurs standards de $G$, c'est l'objet de \cite[Corollary 2.15]{upper}. La proposition suivante permet de montrer comment partiellement reconstruire le motif d'une $G$-variété de drapeau généralisée à partir de son motif sur une extension pondérée.

\begin{prop}\label{calcul}
Soit $X$ une variété de drapeaux généralisée de type $\Theta$ d'un groupe $p$-intérieur et $E/F$ une extension pondérée pour $X$ modulo $p$. Pour toute quasi-$G$-variété de drapeaux $Y$ dominant $X$ modulo $p$ et pour tout entier $i$, on a
$$\chi_{M(X)}(U_Y,i)=\chi_{M(X_E)}(U_{Y_E},i)-\chi_{(M(X)^{<i})_E}(U_{Y_E},i).$$
\end{prop}

\begin{proof}
Observons tout d'abord que la décomposition motivique $M(X)=M(X)^{\geq i}\oplus M(X)^{<i}$ évidente donne lieu à l'égalité
$$\chi_{M(X_E)}(U_{Y_E},i)= \chi_{(M(X)^{\geq i})_E}(U_{Y_E},i)+\chi_{(M(X)^{<i})_E}(U_{Y_E},i).$$

Il s'agit donc de montrer que $\chi_{M(X)}(U_{Y},i)=\chi_{(M(X)^{\geq i})_E}(U_{Y_E},i)$. Considérons un facteur direct $M$ du motif $(M(X)^{\geq i})_E$ qui soit isomorphe à $U_{Y_E}[i]$. Le motif $M$ provient d'un facteur direct motivique indécomposable sur $F$ de $M(X)^{\geq i}$ que l'on note $N$, par Krull-Schmidt.

Le motif $N$ est isomorphe au motif supérieur décalé $U_{Z}[j]$ d'une quasi-$G$-variété de drapeaux généralisée $Z$ qui domine $X$ modulo $p$, par le théorème de structure de Karpenko. Par définition de $M(X)^{\geq i}$ les entier $i$ et $j$ sont visiblement égaux, et quitte à les décaler convenablement, on peut les supposer nuls. Le motif $N_E$ n'est autre que le motif supérieur de la $E$-variété $Z_E$. En particulier, la propriété de Krull-Schmidt sur l'extension $E$ implique que les motifs supérieurs $U_{Y_E}$ et $U_{Z_E}$ sont isomorphes, autrement dit $Y_E$ et $Z_E$ sont équivalentes modulo $p$. L'extension $E/F$ étant supposée pondérée pour $X$ modulo $p$, les $F$-variétés $Y$ et $Z$ sont équivalentes modulo $p$, et leurs motifs supérieurs $U_Y$ et $U_{Z}$ sont de fait isomorphes, d'où l'égalité recherchée.
\end{proof}

\section{Equivalence motivique et motifs supérieurs}

Nous l'avons vu précédemment, la classification des groupes semisimples montre par la détermination des indices de Tits le lien qu'entretiennent les classes d'isomorphismes de groupes semisimples et l'isotropie des variétés de drapeaux généralisées. Le résultat suivant peut être vu comme un analogue motivique de ces résultats et sera essentiel pour construire les données combinatoires caractérisant l'équivalence motivique.\\
La démonstration du critère suivant requiert plusieurs étapes. On traite premièrement des variétés de drapeaux généralisées isotropes, cas reposant essentiellement sur les décompositions de type Levi de Chernousov, Gille et Merkurjev. Le cas général est ensuite déduit du cas isotrope par construction d'une bonne extension pondérée et descente au corps de base. Notons que le passage aux décompositions de type Levi ne préserve que très rarement le caractère absolument simple des groupes, et encore moins leur type. Le résultat suivant doit donc structurellement être démontré pour tous les groupes semisimples à corps de base fixé, indépendamment de leur type.

\begin{thm}\label{classgeoalg}
Soient $G$ et $G'$ deux groupes semisimples sur $F$, tous deux formes intérieures d'un même groupe quasi-déployé. Les groupes $G$ et $G'$ sont motiviquement équivalents modulo $p$ si et seulement s'ils sont équivalents modulo $p$.
\end{thm}

\begin{proof}
Si $G$ et $G'$ sont motiviquement équivalents modulo $p$, les motifs supérieurs standards $U_{\Theta,G}$ et $U_{\Theta,G'}$ sont visiblement isomorphes, pour tout sous-ensemble $\ast$-invariant de leur diagramme de Dynkin. Les variétés de drapeaux généralisées de type fixé sont donc équivalentes modulo $p$.

On démontre le sens réciproque pour l'ensemble des groupes semisimples sur $F$, par récurrence sur le rang commun de $G$ et $G'$,  le résultat étant trivial s'ils sont tous deux de rang nul. Supposons donc $G$ et $G'$ sont de rang strictement positif, et fixons un sous-ensemble $\Theta$ de $\Delta$ invariant sous l'action du groupe de Galois absolu, excluant le cas évident où $\Theta$ est vide.

Il s'agit donc de montrer que les motifs standards $M_{\Theta,G}$ et $M_{\Theta,G'}$ sont isomorphes dans $\Mot(F;\mathbb{F}_p)$. On commence par traiter le cas où $\Theta$ est invariant pour l'action du groupe de Galois absolu, et les variétés de drapeaux généralisées de type $\Theta$ associées à $G$ et $G'$ ont un point rationnel. Dans ce cas d'après théorème de décomposition de type Levi des variétés de drapeaux généralisées \cite[Theorem 7.4]{chermergil}, les motifs $M_{\Theta,G}$ et $M_{\Theta,G'}$ se décomposent en une somme directe de motifs standards de $G_{\Theta}$ et $G'_{\Theta}$, respectivement. En outre, les types de ces motifs ainsi que les décalages qui apparaissent dans ces décompositions ne dépendent que du type de $G$ et $G'$. Les groupes semisimples $G_{\Theta}$ et $G'_{\Theta}$ étant de même type et de rang strictement inférieur à $G$ et $G'$, l'hypothèse de récurrence et la proposition \ref{prop5} assurent qu'ils sont motiviquement équivalents modulo $p$. Les facteurs apparaissant dans les décompositions de type Levi de $M_{\Theta,G}$ et $M_{\Theta,G'}$ sont ainsi isomorphes deux à deux, et les motifs standards de type $\Theta$ de $G$ et $G'$ sont isomorphes.\\

On traite désormais le cas où la variété standard de type $\Theta$ n'a pas de point rationnel. Notons que l'extension composée $E/F$ des corps des fonctions des variétés $X_{\Theta,G}$ et $X_{\Theta,G'}$ est pondérée pour $X_{\Theta,G}$ modulo $p$. En effet, si $Y$ et $Y'$ sont deux variétés quasi-$G$-homogènes équivalentes modulo $p$ sur $E$ et qui dominent $X_{\Theta,G}$ modulo $p$, elles dominent aussi $X_{\Theta,G'}$, puisque $G$ et $G'$ sont équivalents modulo $p$. Elles dominent en particulier la $G\times G'$-variété de drapeaux généralisée $X_{\Theta,G}\times X_{\Theta,G'}$, de par le produit externe des cycles. En vertu de la proposition \ref{pond} l'extension $E/F$ est pondérée pour $X_{\Theta,G}\times X_{\Theta,G'}$, et les variétés $Y_E$ et $Y'_E$ étant équivalentes modulo $p$, il en va de même pour $Y$ et $Y'$.

Puisque les groupes $G_E$ et $G'_E$ possèdent des sous-groupes paraboliques de type $\Theta$, les motifs standards de type $\Theta$ de $G_E$ et $G_E'$ sont isomorphes d'après le le cas isotrope. Supposons par l'absurde que les applications caractéristiques des motifs $M_{\Theta,G}$ et $M_{\Theta,G'}$ diffèrent, en considérant un hypothétique entier $i_0$ minimal tel qu'il existe une quasi-$G$-variété de drapeaux généralisée $Y$ pour laquelle $\Phi_{M_{\Theta,G}}(U_Y,i_0)\neq \Phi_{M_{\Theta,G'}}(U_Y,i_0)$. Le calcul de la proposition \ref{calcul} implique que

\begin{align*}
\Phi_{M_{\Theta,G}}(U_Y,i_0)&=\Phi_{(M_{\Theta,G})_E}(U_{Y_E},i_0)-\Phi_{(M_{\Theta,G}^{<i_0})_E}(U_{Y_E},i_0)\\
&=\Phi_{(M_{\Theta,G'})_E}(U_{Y_E},i_0)-\Phi_{(M_{\Theta,G'}^{<i_0})_E}(U_{Y_E},i_0)\\
&=\Phi_{M_{\Theta,G'}}(U_Y,i_0)
\end{align*}

ce qui contredit la définition de $i_0$. Les applications caractéristiques des motifs $M_{\Theta,G}$ et $M_{\Theta,G'}$ coincident donc, et ces motifs sont isomorphes. Le résultat lorsque $\Theta$ est quelconque se déduisant aisément du cas précédent en appliquant le foncteur de corestriction à une extension $F_{\Theta}/F$ minimale pour laquelle le sous-ensemble $\Theta$ est invariant pour l'action du groupe de Galois absolu sur le diagramme de Dynkin de $G$ et $G'$, ces groupes semisimples sont motiviquement équivalents modulo $p$. 
\end{proof}

Le résultat précédent appliqué pour les groupes intérieurs à tout nombre premier $p$ permet d'obtenir un premier critère d'équivalence motivique.

\begin{cor}\label{inner}
Deux groupes semisimples intérieurs et de même type sont motiviquement équivalents si et seulement si pour tout sous-ensemble de leur diagramme de Dynkin, les variétés standards respectives de type fixé sont équivalentes. 
\end{cor}

Comme expliqué précédemment, les classes d'équivalence modulo $p$ des $G$-variétés de drapeaux généralisées correspondant aux classes d'isomorphismes de motifs supérieurs de $G$ à coefficients dans $\mathbb{F}_p$.

\begin{cor}
La classe d'équivalence motivique modulo $p$ d'un groupe semisimple $G$ est uniquement déterminée par son type et par l'ensemble de motifs supérieurs de $G$ à coefficients dans $\mathbb{F}_p$.
\end{cor}

\section{Indices de Tits supérieurs}

L'objectif de cette section est de construire les versions locales en $p$ des indices de Tits, les \emph{$p$-indices de Tits}, qui fournissent les données combinatoires que nous utilisons par la suite pour caractériser l'équivalence motivique des groupes semisimples.
\begin{defi}
Soit $p$ un nombre premier et $G$ un groupe semisimplesur un corps $F$. Le $p$-indice de Tits de $G$ est la donnée du diagramme de Dynkin de $G$ ainsi que d'un sous-ensemble $\delta_{0}^p(G)$ de ses sommets. Une orbite $\Theta$ de $\Delta(G)$ pour la $\ast$-action est contenue dans $\delta_{0}^p(G)$ si et seulement si une $G$-variété de drapeaux généralisée de type $\Theta$ devient isotrope sur une extension de degré premier à $p$ de $F$.
\end{defi}

L'ensemble $\delta_{0}^p(G)$ des orbites \emph{$p$-distinguées} peut être défini de manière équivalente comme l'union de orbites distinguées apparaissant dans les indices de Tits de $G$ sur toutes les extensions de degré premier à $p$ de son corps de définition. On dit que $G$ est quasi-$p$-déployé lorsque $\Delta(G)$ est lui-même une orbite $p$-distinguée, et le \emph{$p$-indice de Tits supérieur} de $G$ est le foncteur qui associe à toute extension $E/F$ du corps de base le $p$-indice de Tits du groupe semisimple $G_E$.

\begin{thm}\label{classtits}
Soient $G$ et $G'$ deux groupes semisimples sur $F$, tous deux formes intérieures d'un même groupe quasi-$p$-déployé. Les groupes $G$ et $G'$ sont motiviquement équivalents modulo $p$ si et seulement si leur $p$-indices de Tits supérieurs coincident.
\end{thm}

\begin{proof}
Considérons deux groupes semisimples $p$-intérieurs $G$ et $G'$ dont les $p$-indices de Tits supérieurs coincident, et fixons un sous-ensemble $\Theta$ de leur diagramme de Dynkin. Par corestriction d'une extension minimale sur laquelle $\Theta$ est $\ast$-invariant au corps de base, on peut supposer que $\Theta$ est invariant sous l'action du groupe de Galois. Une $G$-variété de drapeaux généralisée de type $\Theta$ possède visiblement un zéro-cycle de degré premier à $p$ sur le corps des fonctions de la variété $X_{\Theta,G'}$, puisque les $p$-indices de Tits de $G$ et $G'$ coincident cette extension. Une $G$-variété de drapeaux généralisée de type $\Theta$ est donc dominée modulo $p$ par une $G'$-variété de drapeaux même type, et ces variétés sont même équivalentes modulo $p$, par le même raisonnement en intervertissant leurs rôles. Les groupes $G$ et $G'$ sont donc équivalents modulo $p$, donc motiviquement équivalents modulo $p$ en vertu du théorème \ref{classgeoalg}.

Montrons désormais que les $p$-indices de Tits supérieurs de deux groupes semisimples $G$ et $G'$ motiviquement équivalents modulo $p$ coincident, c'est à dire que pour toute extension fixée $E/F$ les sous-ensembles $\delta_{0}^p(G_E)$ et $\delta_{0}^p(G'_E)$ son égaux. Par définition, il existe une extension $L/E$ de degré premier à $p$ sur laquelle la variété $X_{\delta_{0}^p(G_E),G}$ est isotrope. En outre, les groupes $G$ et $G'$ étant motiviquement équivalents modulo $p$, les variétés standards de type $\delta_{0}^p(G_E)$ associées à $G$ et $G'$ sont équivalentes modulo $p$, et la variété $X_{\delta_{0}^p(G_E),G'}$ a un zéro-cycle de degré premier à $p$ sur l'extension composée de $L$ et du corps de fonctions de $X_{\delta_{0}^p(G_E),G}$. Puisque $X_{\delta_{0}^p(G_E),G_L}$ est rationnelle, $X_{\delta_{0}^p(G_E),G'_E}$ a un zéro-cycle de degré premier à $p$. On a donc visiblement l'inclusion de $\delta_{0}^p(G_E)$ dans $\delta_{0}^p(G'_E)$, et le même raisonnement interchangeant les rôles de $G$ et $G'$ implique l'égalité de ces sous-ensembles. Les $p$-indices de Tits des groupes semisimples $G_E$ et $G'_E$ sont donc égaux.
\end{proof}

\begin{cor}\label{cortits}
La classe d'équivalence motivique modulo d'un groupe semisimple intérieur est uniquement déterminée par son type et par l'ensemble de ses $p$-indices de Tits supérieurs.
\end{cor}

La classification des groupes semisimples $p$-intérieurs à équivalence motivique modulo $p$ près est ainsi réduite à la détermination des différents $p$-indices de Tits qui apparaissent sur les extensions sur corps de base. On s'attache désormais à décrire dans plusieurs cas classiques les valeurs prises par les $p$-indices de Tits. La détermination des valeurs possibles des $p$-indices de Tits pour tous les groupes semisimples fait l'objet du travail \cite{decskip} en commun avec Skip Garibaldi.

\section{Equivalence motivique, groupes classiques et structures algébriques associées}

Suivant Weil, Tits \cite{tits1,weil}, les groupes semisimples de type classique s'écrivent comme groupes d'automorphismes d'algèbres centrales simples munies d'involutions. On s'attarde ici principalement sur les groupes orthogonaux ainsi que les groupes spéciaux linéaires, les tables complètes de tous les $p$-indices de Tits des groupes semisimples et les différents critères d'équivalence motivique faisant l'objet de \cite{decskip}.

Un groupe semisimple $G$ est quasi-$p$-déployé si n'est pas $p$ un \emph{nombre premier de torsion homologique} de $G$, au sens de \cite{groth2,tits2}. Les expression des $p$-indices de Tits en fonction des invariants classiques associés aux structures algébriques sous-jacentes aux groupes semisimples nous permettent de produire des critères d'équivalence motivique purement algébriques. Ils clarifient en outre le lien qu'entretiennent les motifs et la géométrie birationnelle des variétés de drapeaux généralisées. \\

Par soucis de simplicité, on se concentre principalement dans cette partie sur les groupes de type intérieur. D'une part, deux groupes strictement isogènes étant motiviquement équivalent, il est loisible de limiter l'étude aux groupes simplement connexes (on pourrait faire de même en remplacant \emph{simplement connexe} par \emph{adjoint}). Un tel groupe est alors isomorphe à un produit fini de groupes absolument simples, eux mêmes simplement connexes. En outre, par définition de l'équivalence motivique (ou par le théorème \ref{classtits}, au choix) deux produits $G_1\times ... \times G_n$ et $G'_1 \times ... \times G'_n$ de groupes absolument simples sont motiviquement équivalents modulo $p$ si et seulement si les facteurs le sont deux à deux, à permutation près. 

\subsection{Groupes intérieurs de type $A_n$} Rappelons qu'une $F$-algèbre simple est dite \emph{centrale} si elle est de dimension finie et de centre $F$. Les groupes absoluments simples simplement connexes intérieurs de type $A_n$ sont les groupes spéciaux linéaires associés aux algèbres simples centrales (les groupes projectifs linéaires étant leurs analogues adjoints).\\

Le \emph{groupe de Brauer de $F$} est le groupe obtenu à partir de l'ensemble des classes de similitudes de $F$-algèbres simples centrales, muni de la structure induite par le produit tensoriel. On notera dans la suite $[A]$ la classe de similitude d'une $F$-algèbre simple centrale $A$. \`{A} isomorphisme près, toute algèbre simple centrale $A$ est semblable à une unique algèbre à division, dont le degré (c'est à dire la racine carrée de la dimension) est l'\emph{indice} de $A$. 

Soit donc $A$ une $F$-algèbre simple centrale et $p$ un nombre premier. On note $d_p$ la plus grande puissance de $p$ divisant l'indice de $A$. Le $p$-indice de Tits du groupe spécial linéaire $\Sl_1(A)$ est de la forme suivante, suivant les tables de \cite{decskip}.\\

\begin{figure}[!h] 
\includegraphics[width=\textwidth]{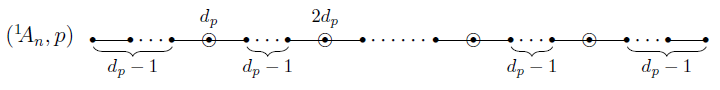}
\end{figure}

Le théorème \ref{classtits} assorti à cette description permet donc de caractériser l'équivalence motivique des groupes spéciaux linéaire comme suit.

\begin{crit}\label{critsev}Deux groupes spéciaux linéaires $\Sl(A)$ et $\Sl(B)$ sont motiviquement équivalents modulo $p$ si et seulement si pour toute extension $E/F$, les valuations $p$-adiques des indices de $A_E$ et $B_E$ coincident.
\end{crit}

Les $\Sl(A)$-variétés de drapeaux généralisées sont les variétés de drapeaux d'idéaux de $A$ de dimension fixée. Suivant nos conventions, une $\Sl(A)$-variété de drapeaux généralisée de type $\{i_1,...,i_k\}$ associée à une suite strictement croissante d'entiers est isomorphe à la variété des drapeaux d'idéaux à droite de $A$ de dimension réduite $i_1,...,i_k$.

\begin{thm}\label{motequivsevbrauer}
Deux groupes spéciaux linéaires $\Sl(A)$ et $\Sl(B)$ sont motiviquement équivalents si et seulement si les variétés de Severi-Brauer $\SB(A)$ et $\SB(B)$ sont stablement birationnellement équivalentes. Ces conditions sont équivalentes à l'égalité des indices de Tits de $\Sl(A)$ et $\Sl(B)$ sur toute extension du corps de base.
\end{thm}

\begin{proof}
Les groupes $\Sl(A)$ et $\Sl(B)$ sont motiviquement équivalents si et seulement si leurs indices de Tits coincident sur toutes extension de corps, c'est le corollaire \ref{cortits}. Cette condition est équivalente, c'est bien connu, à l'équivalence birationnelle stable des variétés de Severi Brauer associées. 
\end{proof}

Les classes d'équivalences motiviques des groupes intérieurs de type $A_n$ sont donc déterminées par les indices de Tits classiques. Comme observé précédemment, le lien posé entre l'équivalence motivique de ces groupes et les classes d'équivalence birationnnelle stables de variétés de Severi-Brauer ne vaut absolument pas si l'on considère l'équivalence motivique à coefficients entiers. On peut en effet montrer que deux tels groupes sont motiviquement équivalents à coefficients entiers si et seulement s'ils sont isomorphes. En particulier on peut construire des groupes semisimples intérieurs de type $A_n$ qui ne sont \emph{pas} motiviquement équivalents à coefficients entiers, alors même que les variétés de Severi-Brauer associées sont \emph{birationnellement équivalentes}.

Avant de nous attarder sur les groupes orthogonaux, observons deux conséquences frappantes des critères d'équivalence motivique des groupes spéciaux linéaires.

\begin{thm}\label{anisotrope}
Soient $A$ et $B$ deux algèbres simples centrales de même degré, et deux variétés de drapeaux généralisées $X_{\Theta,\Sl_1(A)}$ et $X_{\Theta,\Sl_1(B)}$ anisotropes et de même type. Les conditions suivantes sont équivalentes :
\begin{enumerate}
 \item les motifs supérieurs $U_{\Theta,\Sl_1(A)}$ et $U_{\Theta,\Sl_1(B)}$ à coefficients dans un corps fini à $p$ éléments sont isomorphes;

\item les variétés $X_{\Theta,\Sl_1(A)}$ et $X_{\Theta,\Sl_1(B)}$ sont équivalentes modulo $p$;

\item les motifs $M_{\Theta,\Sl_1(A)}$ et $M_{\Theta,\Sl_1(B)}$ à coefficients dans un corps à $p$ éléments sont isomorphes;

\item les groupes $\Sl_1(A)$ et $\Sl_1(B)$ sont motiviquement équivalents modulo $p$.
\end{enumerate}
\end{thm}

\begin{proof}
Il suffit en vertu des discussions précédentes de montrer que la condition $(1)$ implique la condition $(4)$. Supposons donc que les motifs supérieurs $U_{\Theta,\Sl_1(A)}$ et $U_{\Theta,\Sl_1(B)}$ à coefficients dans un corps à $p$ éléments soient isomorphes. Le théorème de structure de Karpenko implique que ces motifs sont isomorphes respectivement à deux motifs supérieurs $U_{\{p^k\},\Sl_1(D)}$ et $U_{\{p^l\},\Sl_1(D')}$, où $D$ et $D'$ sont deux algèbres à division semblabes aux parties $p$-primaires de $A$ et $B$. La dichotomie motivique des groupes projectifs linéaires \cite[Theorem 4.2, Corollary 4.3]{dec3} induit donc d'une part que $k$ et $l$ sont égaux, et d'autre part que les motifs supérieurs de type fixé de $\Sl_1(D)$ et $\Sl_1(D')$ sont isomorphes deux à deux. Il en va donc de même pour $\Sl_1(A)$ et $\Sl_1(B)$ et le théorème \ref{classgeoalg} permet d'affirmer que ces groupes sont donc motiviquement équivalents modulo $p$.
\end{proof}

\begin{cor}
Soient $A$ et $B$ deux algèbres simples centrales non-déployées et de même degré. Les groupes $\Sl_1(A)$ et $\Sl_1(B)$ sont motiviquement équivalents si et seulement si deux variétés de drapeaux anisotropes $X_{\Theta,\Sl_1(A)}$ et $X_{\Theta,\Sl_1(B)}$ sont stablement birationnellement équivalentes.
\end{cor}

\begin{proof}
La condition nécessaire est contenue dans le corollaire \ref{inner}. En outre, si deux variétés $X_{\Theta,\Sl_1(A)}$ et $X_{\Theta,\Sl_1(B)}$ sont stablement birationnellement équivalentes, elle sont à fortiori équivalentes modulo tout nombre premier $p$. La condition suffisante découle donc du théorème \ref{anisotrope}.
\end{proof}

\subsection{Groupes orthogonaux}

Le groupe spécial orthogonal $\SO(q)$ d'une forme quadratique non-dégénérée sur un corps de caractéristique différente de $2$ est, suivant la parité de la dimension de $q$, de type $B_n$ ou $D_n$, éventuellement extérieur. Comme expliqué précédemment, les groupes orthogonaux peuvent être remplacés par leurs revêtements simplement connexes (c'est à dire les groupes spinoriels) dans l'ensemble des énoncés suivants.

Si deux formes quadratiques sont de même dimension impaire, ou si leur dimension est paire et leur discriminant est de même nature, l'équivalence motivique des groupes orthogonaux associés peut se déduire directement du théorème \ref{classtits} et des indices suivants, extraits de \cite{decskip}.

\begin{figure}[!h] 
\includegraphics[width=\textwidth]{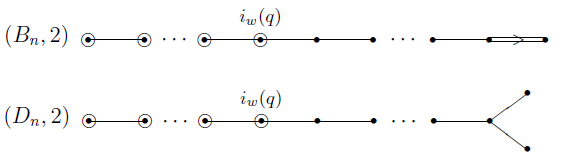}
\end{figure}

Un petit travail supplémentaire est cependant nécessaire si l'on veut éviter les circonvolutions concernant le type des groupes considérés dans le cas $D_n$.

Suivant \cite{decskip}, on se concentre ici sur $2$-indices de Tits ($2$ étant l'unique nombre premier de torsion homologique des groupes orthogonaux). Observons $2$-indice de Tits d'un groupe spécial orthogonal correspond à l'indice de Tits classique et est déterminé par l'indice de Witt de la forme quadratique sous-jacente, c'est le theorème de Springer.

\begin{crit}\label{critquad}Soient $q$ et $q'$ deux formes quadratiques de même dimension. Les groupes orthogonaux $\SO(q)$ et $\SO(q')$ sont motiviquement équivalents si et seulement si les indices de Witt de $q$ et $q'$ coincident sur toutes les extensions de leur corps de définition. Ces conditions sont équivalentes à l'isomorphisme des motifs des quadriques projectives associés à $q$ et $q'$ (à coefficients entiers ou dans $\mathbb{F}_2$, au choix).
\end{crit}

\begin{proof}
La première équivalence est claire pour deux groupes orthogonaux de même type, d'après la discussion précédente. Pour démontrer le cas général, on observe que si les indices de Witt de $q$ et $q'$ coincident sur toute extension du corps de base, les groupes $\SO(q)$ et $\SO(q')$ sont nécessairement formes intérieures d'un même groupe quasi-déployé, c'est l'objet de \cite[Lemma 2.6]{motequivkarp}.

Sous ces conditions les motifs des quadriques projectives de $q$ et $q'$ sont isomorphes à coefficients dans $\mathbb{F}_2$. Cet isomorphisme se relève en un isomorphisme des motifs de ces quadriques à coefficients entiers par \cite{hau1}. Le critère de Vishik \cite[Criterion 0.1]{motequivkarp} assure donc que les indices de Witt de $q$ et $q'$ coincident sur toutes les extensions de leur corps de définition.
\end{proof}

Notons que sous certaines hypothèses supplémentaires les conditions du critère \ref{critquad} sont équivalentes au fait que les groupes $\SO(q)$ et $\SO(q')$ (ou les quadriques projectives de $q$ et $q'$, c'est équivalent) soient isomorphes. C'est par exemple en dimension impaire et en basse dimension \cite{izh}, ou sur les corps globaux \cite{hoffman2}.

Comme signalé par A. Vishik, le critère précédent peut être partiellement retrouvé à partir des résultats de \cite{vishmot} sous l'hypothèse que le corps de base soit de caractéristique $0$ et en se restreignant aux variétés de drapeaux de sous-espaces isotropes de dimensions consécutives $0,...,k$. En effet les décompositions de \cite[Claim 3.2]{vishmot} obtenues dans la catégorie des motifs de Voevodsky permettent de montrer que si les motifs de deux quadriques à coefficients entiers sont isomorphes, il en va de même pour les motifs de drapeaux consécutifs de sous espaces isotropes. Les motifs supérieurs considérés dans cet article correspondent alors dans ce cas particulier aux schémas simpliciaux $\mathcal{X}_{Q^i}$ de Vishik.

Notons enfin que le lien entre la géométrie birationnelle des quadriques et l'équivalence motivique des groupes orthogonaux est moins riche que celui observé pour les variétés de Severi-Brauer au théorème \ref{motequivsevbrauer}. On peut construire des formes quadratiques dont les quadriques respectives sont \emph{birationnellement équivalentes}, sans pour autant que les groupes orthogonaux associés soient motiviquement équivalents (voir \cite{ohm,rouss}). L'équivalence motivique est donc un invariant trop fin pour être utilisé autour du problème de Zariski sur les quadriques \cite[Chap. XIII, Question 6.10]{lam,hoffman}, contrairement au problème d'Amitsur des variétés de Severi-Brauer \cite{amitsur}.

\vspace{1cm}

\bibliographystyle{amsplain}

\begin{thebibliography}{10}

\bibitem{ohm} \textsc{Ahmad, H., Ohm, J.} \emph{Function fields of Pfister neighbors}, J. Algebra 178, 653-664, 1995.

\bibitem{amitsur} \textsc{Amitsur, S.}, \emph{Generic splitting fields of central simple algebras}, Ann. of Math. 62, 8-43, 1955.

\bibitem{borel} \textsc{Borel, A.}, \emph{Linear Algebraic groups}, second ed., Graduate Texts in Mathematics, vol. 126, Springer-Verlag, New-York, 1991.

\bibitem{borelspringer} \textsc{Borel, A., Springer, T.A.}, \emph{Rationality properties of linear algebraic groups II}, Tôhoku Math. J. 20, 443-497, 1968.

\bibitem{boreltits} \textsc{Borel, A., Tits, J.}, \emph{Groupes réductifs}, Publications mathématiques de l'I.H.É.S, Tome 27, 55-151, 1965.

\bibitem{bourba} \textsc{Bourbaki, N.} \emph{Éléments de mathématique}, Groupes et algèbres de Lie, Chapitres 4, 5 et 6, Masson, Paris, 1981.

\bibitem{brosnan} \textsc{Brosnan, P.}, \emph{On motivic decompositions arising from the method of Bia lynicki-Birula}, Invent. MAth. 161, 1, 91-111, 2005.

\bibitem{chowmottwistflag} \textsc{Calmès, B., Semenov, N., Petrov, V., Zainoulline, K.}, \emph{Chow motives of twisted flag varieties}, 
Compos. Math. 142, 1063-1080, 2006.

\bibitem{chermergil} \textsc{Chernousov, V., Gille, S., Merkurjev, A.} \emph{Motivic decomposition of isotropic projective homogeneous varieties}, Duke Math. J., 126(1) :137-159, 2005.

\bibitem{chermer} \textsc{Chernousov, V., Merkurjev, A.} \emph{Motivic decomposition of projective homogeneous varieties and the Krull-Schmidt theorem}, Transformation Groups 11, 371-386, 2006.

\bibitem{decvarprojcoeff} \textsc{De Clercq, C.} \emph{Motivic decompositions of projective homogeneous varieties and change of coefficients}, C. R. Math. Acad. Sci. Paris 348, volume 17-18, 955--958, 2010.

\bibitem{dec3} \textsc{De Clercq, C.} \emph{Classification of upper motives of algebraic groups of inner type $A_n$}, C. R. Math. Acad. Sci. Paris 349, volume 7-8, 433--436, 2011.

\bibitem{decskip} \textsc{De Clercq, C., Garibaldi, S.} \emph{On the Tits $p$-indexes of semisimple algebraic groups}, preprint.

\bibitem{groth} \textsc{Demazure, M., Grothendieck, A., \emph{Schémas en groupes. II : Groupes de type multiplicatif, et structure des schémas en groupes généraux. Séminaire de Géométrie Algébrique du Bois-Marie (SGA 3). Dirigé par M. Demazure et A. Grothendieck. Lecture Notes in Mathematics, Vol. 152. Springer-Verlag, Berlin, 1962/64}}

\bibitem{EKM} \textsc{Elman, R., Karpenko, N., Merkurjev, A.} \emph{The Algebraic and Geometric Theory of Quadratic Forms}, AMS Colloquium Publications, Vol. 56, 2008.
\bibitem{fulton} \textsc{Fulton, W.} \emph{Intersection theory}, Ergebnisse der Mathematik und ihrer Grenzgebiete. 3. Folge. A Series of Modern Surveys in Mathematics 2, Springer-Verlag, 1998.

\bibitem{manin} \textsc{Manin, Yu.I.} \emph{Correspondence, Motifs, and Monoidal Transformations}, Mat. Sb., Nov. Ser. 77 (1968) 475-507, English transl. : Math.USSR, Sb. 6 439-470, 1968.

\bibitem{shells} \textsc{Garibaldi, S., Petrov, V., Semenov, N.}, \emph{Shells of twisted flag varieties and the Rost invariant}, preprint 2013, disponible sur la page personnelle des auteurs.

\bibitem{groth2} \textsc{Grothendieck, A.}, \emph{La torsion homologique et les sections rationnelles}, Anneaux de Chow et applications, Exposé 5, Séminaire C. Chevalley, 1958.

\bibitem{hau1} \textsc{Haution, O.}, \emph{Lifting of coefficients for Chow motives of quadrics}, Quadratic forms, linear algebraic groups, and cohomology, Dev. Math. 18, 239-247, Springer, New York, 2010.

\bibitem{hoffman} \textsc{Hoffmann, D.W.}, \emph{Similarity of quadratic forms and half-neighbors}, J. Algebra 204, 255-280, 1998.

\bibitem{hoffman2} \textsc{Hoffmann, D.W.}, \emph{Motivic equivalence en similarity of quadratic forms}, preprint.

\bibitem{izh} \textsc{Izhboldin, O.}, \emph{Motivic equivalence of quadratic forms}, Documenta Mathematica vol. 3, 341-351, 1998.

\bibitem{kahn} \textsc{Kahn, B.}, \emph{Formes quadratiques et cycles algébriques [d'après Rost, Voevodsky, Vishik, Karpenko...]}, Exposé Bourbaki no 941, Astérisque 307, 113-163, 2006.

\bibitem{kahn2} \textsc{Kahn, B.}, \emph{Formes quadratiques sur un corps}, Cours spécialisé no 15, SMF, 2008.

\bibitem{ICM} \textsc{Karpenko, N.}, \emph{Canonical dimension}, Proceedings of the ICM 2010, vol. II, 146-161.

\bibitem{motequivkarp} \textsc{Karpenko, N.}, \emph{Criteria of motivic equivalence for quadratic forms and central simple algebras}, Math. Ann. 317 no. 3, 585-611, 2000.

\bibitem{essdimquad} \textsc{Karpenko, N., Merkurjev, A.}, \emph{Esential dimension of quadrics}, Invent. Math. 153, 361-372, 2003.

\bibitem{kar1} \textsc{Karpenko, N.}, \emph{Sufficiently generic orthogonal grassmannians}, J. Algebra 372, 365--375, 2012. 

\bibitem{upper} \textsc{Karpenko, N.}, \emph{Upper motives of algebraic groups and incompressibility of Severi-Brauer varieties}, J. Reine Angew. Math. 677, 179--198, 2013.

\bibitem{outer} \textsc{Karpenko, N.}, \emph{Upper motives of outer algebraic groups}, Quadratic forms, linear algebraic groups, and cohomology, Dev. Math. 18, 249-258, Springer, New York, 2010.

\bibitem{kerreh} \textsc{Kersten, I., Rehmann, U.} \emph{Generic splitting of reductive groups}, Tohoku Math. J., vol. 46 1, 35-70, 1994.

\bibitem{KMRT} \textsc{Knus, M.-A., Merkurjev, A., Rost, M., Tignol, J.-P.}, \emph{The book of involutions}, préface de J. Tits, AMS Colloquium Publications, Vol. 44, 1998.

\bibitem{kock} \textsc{K\"{o}ck, B.} \emph{Chow motif and higher Chow theory of $G/P$}, Manuscripta Math., 70(4) :363-372, 1991.

\bibitem{lam} \textsc{Lam, T.Y.} \emph{Introduction to quadratic forms over fields}, Graduate Studies in Mathematics, vol. 67, American Mathematical Society, Providence, RI, 2005.


\bibitem{index} \textsc{Merkurjev, A., Panin, A., Wadsworth, A.} \emph{Index reduction formulas for twisted flag varieties, I}, Journal K-theory 10, 517-596, 1996.

\bibitem{index2} \textsc{Merkurjev, A., Panin, A., Wadsworth, A.} \emph{Index reduction formulas for twisted flag varieties, II}, Journal K-theory 14, 101-196, 1998.

\bibitem{jinv} \textsc{Petrov, V., Semenov, N., K.Zainoulline} \emph{J-invariant of linear algebraic groups}, Ann. Sci. Éc. Norm. Sup. 41, 1023-1053, 2008.

\bibitem{gensplit1} \textsc{Petrov, V., Semenov, N.} \emph{Generically split projective homogeneous varieties}, avec un appendice de M. Florence, Annales Scientifiques de l'ecole normale supérieure 41, 1023-1053, 2008.

\bibitem{gensplit2} \textsc{Petrov, V., Semenov, N.} \emph{Generically split projective homogeneous varieties. II}, J. of K-theory 10, issue 1, 1-8, 2010.

\bibitem{trial} \textsc{Quéguiner-Mathieu, A., Semenov, N., Zainoulline, K.}, Journal of Pure and Applied Algebra 216, issue 12, 2614-2628, 2012.

\bibitem{rost1} \textsc{Rost, M.} \emph{Some new results on the Chow groups of quadrics}, preprint available on the web page of the author, 1990.

\bibitem{rost2} \textsc{Rost, M.} \emph{The motive of a Pfister form}, preprint available on the web page of the author, 1998.

\bibitem{rouss} \textsc{Roussey, S.} \emph{Corps de fonctions et équivalences birationnelles des formes quadratiques}, Thèse de doctorat, Université de Franche-Comté, 2005.


\bibitem{springer} \textsc{Springer, T.A.}, \emph{Sur les formes quadratiques d'indice zéro}, C.R. Acad. Sci. Paris 234, 1517-1519, 1952.


\bibitem{tits1} \textsc{Tits, J.}, \emph{Classification of algebraic semisimple groups}, Algebraic Groups and Discontinuous Sibgroups, Amer. Math. Soc., Providence, RI, 33-62, 1966.

\bibitem{tits2} \textsc{Tits, J.}, \emph{Sur les degrés des extensions de corps déployant les groupes algébriques simples}, C.R. Acad. Sci. Paris 315, Série 1, 1131-1138, 1992.


\bibitem{vishyag} \textsc{Vishik, A., Yagita, N.}, \emph{Algebraic cobordism of a Pfister quadric}, J. of the London Math. Soc. 76 no2, 586-604, 2007.

\bibitem{vishmot} \textsc{Vishik, A.}, \emph{Integral Motives of Quadrics}, preprint of the Max Planck Institute 13, available at www.maths.nottingham.ac.uk/personal/av/Papers/preprint-287.pdf, 1998.

\bibitem{vish} \textsc{Vishik, A.} \emph{Motives of quadrics with applications to the theory of quadratic forms}, Lect. Notes in Math. 1835, Proceedings of the Summer School "Geometric Methods in the Algebraic Theory of Quadratic Forms, Lens 2000", 25-101, 2004.


\bibitem{weil} \textsc{Weil, A.} \emph{Algebras with involutions and the classical groups}, J. Indian Math. Soc. (N.S.) 24, 589-623, 1961


\end{thebibliography}

\end{document}